\documentclass[a4paper,10pt]{article}

\usepackage[latin5]{inputenc}
\usepackage{amsfonts}
\usepackage{amssymb}
\usepackage{graphicx}
\usepackage{amsmath}
\usepackage{pstricks}
\usepackage{psfrag}
\usepackage{cite}
\usepackage{bbm}
\usepackage{amsthm}
\usepackage{pstricks,pst-node,pst-tree}
\usepackage[all]{xy}
\usepackage{amsfonts}
\usepackage{amssymb}
\usepackage{graphicx}
\usepackage{amsmath}
\usepackage{pstricks}
\usepackage{psfrag}
\usepackage{xypic}
\usepackage{cite}
\usepackage[all]{xy}
\usepackage{lscape}
\usepackage{bbm}
\newcounter{rc}

\newcounter{theorem}

\newtheorem{definition}[rc]{Definition}
\newtheorem{example}[rc]{Example}

\newtheorem{proposition}[rc]{Proposition}

\makeatletter
\renewcommand{\thefigure}{\arabic{section}.\arabic{figure}} 
\renewcommand{\therc}{\arabic{section}.\arabic{rc}} 

\begin{document}

\title{\textbf{ON $p-$ADIC INTEGERS AND THE ADDING MACHINE GROUP}}
\author{B\"{u}nyamin DEM\.{I}R\footnote{Department of Mathematics, Anadolu University, 26470, Eski\c{s}ehir, Turkey, e-mails: bdemir@anadolu.edu.tr} \  and Mustafa SALTAN\footnote{Department of Mathematics, Anadolu University, 26470, Eski\c{s}ehir, Turkey, e-mails: mustafasaltan@anadolu.edu.tr}}
\date{}

\maketitle

\thispagestyle{empty}
\begin{abstract}
In this paper, we define a natural
metric on $Aut(X^{*})$ and
prove that the closure of the adding machine group, a subgroup of the automorphism group, is both isometric and isomorphic to the group of $p-$adic integers. So, we show that the group of $p-$adic integers can be isometrically embedded into the metric space
$Aut(X^{*})$.
\end{abstract}

\pagestyle{myheadings}
\let\thefootnote\relax
\let\thefootnote\relax
\begin{center}

\section{Introduction}
\end{center}
In recent years, there are many works on self-similar automorphism groups of the rooted tree $X^{*}$ (\cite{LGN}, \cite{Gri}, \cite{Nekrashevych}). The adding machine group is a typical example for self-similarity. We denote this group by $A$. $A$ is a cyclic group generated by
\begin{equation*}
a=(\underset{p-1 \ \text{times}}{\underbrace{1,1,\ldots,1}},a)\sigma
\end{equation*}%
where $a$ is an automorphism of the $p-$ary rooted tree and $\sigma =(012\ldots (p-1))$  is a permutation on $X=\{0,1,2,\ldots,(p-1)\}$. Thus, $A$ is isomorphic to $\mathbb{Z}$. On the other hand, one can consider the automorphism $a$ as adding one to a $p-$adic integer. That is why the term adding machine is used (\cite{Gri}). In \cite{Holly}, $p-$adic integers is pictured on a tree. This picture serves that any ultrametric space can be drawn on a tree.

In this paper, we equip $Aut(X^{*})$ with a natural metric and prove that the group of $p-$adic integers is both isometric and isomorphic to the closure $\overline{A}$ of the adding machine group, a subgroup of the automorphism group of the $p-$ary rooted tree.

First we recall basic definitions and notions.

\noindent
\textit{$p-adic$ integers:}
A  $p-$adic integer is a formal series
\begin{equation*}
\sum_{i\geq0}a_{i}p^{i}
\end{equation*}
where each $a_{i}\in \{0,1,2,\ldots, (p-1)\}$ and the set of all  $p-$adic integers is denoted by $\mathbb{Z}_{p}$.

Suppose that $a=\sum_{i\geq 0}a_{i}p^{i}$ and $b=\sum_{i\geq 0}b_{i}p^{i}$ be
elements of $\mathbb{Z}_{p}$. Then $a$ addition with $b$, $c=\sum_{i\geq 0}c_{i}p^{i}$,  is determined for each $m\in \{0,1,2,\ldots\}$ by
\begin{equation*}
\sum_{i= 0}^{m}c_{i}p^{i}\equiv\sum_{i= 0}^{m}(a_{i}+b_{i})p^{i} \ \
\ (mod\ p^{m+1})
\end{equation*}
where $c_{i}\in \{0,1,\ldots,(p-1)\}$.  $\mathbb{Z}_{p}$ is a group under this operation and is called the group of $p-$adic integers.

Let $a=\sum_{i\geq0}a_{i}p^{i} $ be an element of $\mathbb{Z}%
_{p}$ and $a\neq0$. Thus, there is a first index $v(a)\geq 0$ such that $%
a_{v}\neq0$. This index is called the order of $a$ and is
denoted by $ord_{p}(a)$. If $a_{i}=0$ for $i=0,1,2,\ldots$ then $%
ord_{p}(a)=\infty$. On the other hand, the $p-$adic value of $a$ is denoted by
\begin{equation*}
|a|_{p}=\left\{
\begin{array}{lll}
0 & , & \text{if} \ a_{i}=0 \ \text{for} \ i=0,1,2,\ldots, \\
p^{-ord_{p}(a)} & , & \text{otherwise}%
\end{array}%
\right.
\end{equation*}
and $d_{p}=|a-b|_{p}$ for $a,b \in \mathbb{Z}_{p}$ is a metric on $ \mathbb{Z}_{p}$ (for details see \cite{Fernando}, \cite{Robert} and \cite{Sch}).

\noindent
\textit{The automorphism group of the rooted tree:}
Let $X$ be a finite set
(alphabet) and let
\begin{equation*}
\begin{array}{c}
X^{\ast }=\{x_{1}x_{2}\ldots x_{n} \text{ }|\text{ }x_{i}\in X,n\geqslant 0\text{
}\}%
\end{array}%
\end{equation*}%
be the set of all finite words. The length of a word $v=x_{1}x_{2}\ldots x_{n}\in
X^{\ast }$ is the number of its letters and is denoted by $|v|$. The product of
$v_{1},v_{2}\in X^{\ast }$ is naturally defined by concatenation $v_{1}v_{2}$%
. One can think of $X^{\ast }$ as vertex set of a rooted tree.
\begin{figure}[h]
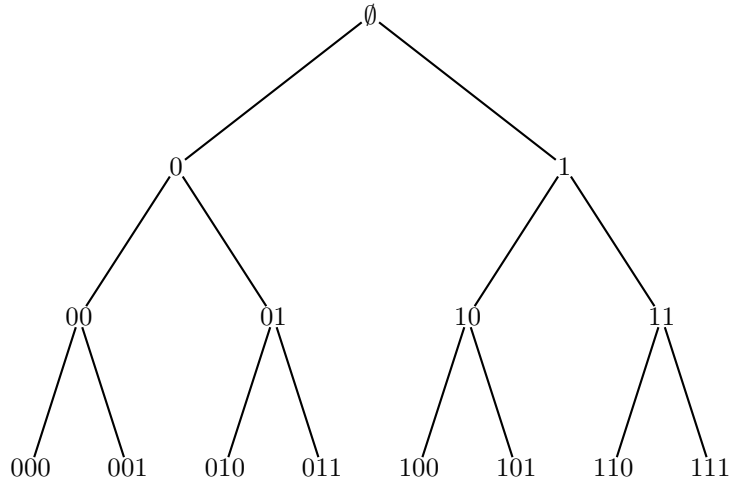

\label{tree} \centering
\pstree[nodesep=0.75pt]{\TR[name=R]{$\emptyset$} }{
       \pstree{ \TR{0} }{
                    \pstree{ \TR{00} } {
                                 \TR[name=d3]{000}
                                 \TR{001}
                              }
                    \pstree{ \TR{01} } {
                                 \TR{010}
                                 \TR{011}
                              }
                              }
        \pstree{ \TR{1} }{
                   \pstree { \TR{10} } {
                                \TR{100}
                                \TR{101}
                               }
                   \pstree { \TR{11} } {
                                \TR{110}
                                \TR[name=d2]{111}
                               }
                               }
                         }
\caption{The first three levels of the binary rooted tree $X^{\ast }$ for $%
X=\{0,1\}$}
\end{figure}

The set $X^{n}=\{v\in X^{\ast } \text{ }| \text{ } |v|=n \}$ is called the $nth$ level of $X^{\ast }$. The empty word $\emptyset$ is the root
of the tree $X^{\ast }$. Two words are connected by an edge if and only if they
are of the form $v,vx $ where $v\in X^{\ast }$ and $x\in X$.

A map $f:X^{\ast }\rightarrow X^{\ast }$ is an endomorphism of the tree $%
X^{\ast }$ if it preserves the root and adjacency of the vertices. An
automorphism is a bijective endomorphism. The group of all automorphisms of
the tree $X^{\ast }$ is denoted by $Aut(X^{\ast })$.

If $G\leq Aut(X^{\ast })$ is an automorphism group of the rooted tree $%
X^{\ast }$ then for $v\in X^{\ast }$, the subgroup
\begin{equation*}
G_{v} =\{g\in G\text{ }|\text{ } g(v)=v\}
\end{equation*}
is called the vertex stabilizer. The $nth$ level
stabilizer is the subgroup
\begin{equation*}
St_{G}(n) =\bigcap_{v\in X^{n}} G_{v}.
\end{equation*}
We need a useful way to express automorphisms of the rooted tree $ X^{\ast }$. For this aim, we give a definition and a proposition from \cite{Nekrashevych}.
\begin{definition}[\cite{Nekrashevych}]
Let $H$ be a group acting (from the right) by permutations on a set $X$ and let $G$ be an arbitrary group. Then the (permutational) wreath product $G\wr H$ is the semi-direct product $G^{X}\rtimes H$, where $H$ acts on the direct power $G^{X}$ by the respective permutations of the direct factors.
\end{definition}
Let $|X|=d$. The multiplication rule for the elements $(g_{1},g_{2},...,g_{d})h\in G\wr H$ is given by the formula
\begin{equation*}
(g_{1},g_{2},...,g_{d})\alpha(h_{1},h_{2},...,h_{d})\beta=(g_{1}h_{\alpha(1)},g_{2}h_{\alpha(2)},...,g_{d}h_{\alpha(d)})\alpha \beta
\end{equation*}
where $g_{i},h_{i}\in G,\alpha,\beta\in H$ and $\alpha(i)$ is the image of $i$ under the action of $\alpha$.
\begin{proposition}[\cite{Nekrashevych}]
Denote by $S(X)$ the symmetric group of all permutations of $X$. Fix some indexing $\{x_{1},x_{2},...,x_{d}\}$ of $X$. Then we have an isomorphism
\begin{equation*}
\psi:Aut(X^{\ast })\rightarrow Aut(X^{\ast})\wr S(X),
\end{equation*}
given by
\begin{equation*}
\psi(g)=(g|_{x_{1}},g|_{x_{2}},...,g|_{x_{d}})\alpha,
\end{equation*}
where $\alpha$ is the permutation equal to the action of $g$ on $X\subset X^{\ast}.$
\end{proposition}
Thus, $g\in Aut(X^{\ast })$ is identified with the image $\psi(g)\in Aut(X^{\ast})\wr S(X)$ and it is written as
\begin{equation*}
g=(g|_{x_{1}},g|_{x_{2}},...,g|_{x_{d}})\alpha.
\end{equation*}

\noindent
\textit{The adding machine group:}
Let $a$ be the transformation on $X^{\ast }$ defined by the wreath recursion%
\begin{equation*}
a=(\underset{p-1 \ \text{times}}{\underbrace{1,1,\ldots,1}},a)\sigma
\end{equation*}%
where $\sigma =(012\ldots (p-1))$ is an element of the symmetric group on $X=\{0,1,2,\ldots,(p-1)\}$. The transformation $a$ generates an infinite cyclic group on $X^{\ast }$. This group is called the adding machine group and we denote this group by $A$.
\begin{figure}[h]
\centering \includegraphics[scale=0.75]{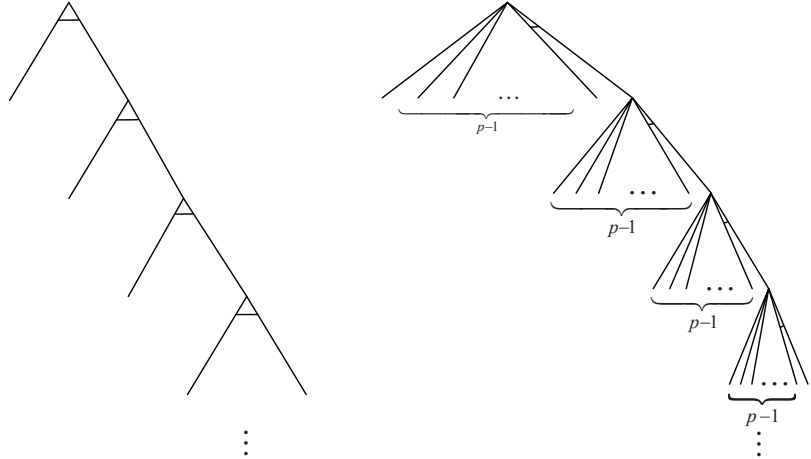}
\caption{Portrait of the transformation $a$ for $X=\{0,1\}$ and $X=\{0,1,...,p-1\}$}
\end{figure}

For example, using permutational wreath product we obtain that%
\begin{equation*}
\begin{array}{ccl}
a^{p} & = & (1,\ldots,1,a)\sigma (1,\ldots,1,a)\sigma \ldots(1,\ldots,1,a)\sigma \\
& = & (a,a,\ldots,a)\sigma ^{p} \\
& = & (a,a,\ldots,a)%
\end{array}%
\end{equation*}%
(for details see \cite{LGN}, \cite{Nekrashevych}).
\begin{center}
\renewcommand{\thefigure}{\arabic{section}.\arabic{figure}} \setcounter{figure}{0}
\renewcommand{\therc}{\arabic{section}.\arabic{rc}} \setcounter{rc}{0}

\section{The Metric Space ($Aut(X^{\ast })$, $d$)}
\end{center}
We define a natural metric on the automorphism group of the $p-$ary rooted tree $X^{\ast}$
where $X=\{0,1,2,...,p-1\}$. This metric is used by \cite{Sunic}.
\begin{definition}
\label{metric} Let $g_{1},g_{2}\in Aut(X^{\ast })$.%
\begin{equation*}
d(g_{1},g_{2})=\left\{
\begin{array}{lll}
\frac{1}{p^{k}} &  & \text{for} \ g_{1}^{-1}g_{2}\in St_{Aut(X^{\ast })}(k) \ \text{and} \ g_{1}^{-1}g_{2}\notin
St_{Aut(X^{\ast })}(k+1), \\
0 &  & \text{for} \ g_{1}=g_{2}.%
\end{array}%
\right.
\end{equation*}%
In other words, if $g_{1}$ and $g_{2}$ agree on all vertices of level $k$
but do not agree at least one vertex of level $(k+1)$ of the tree $X^{\ast }$
then the distance between $g_{1}$ and $g_{2}$ is $\frac{1}{p^{k}}$.
\end{definition}

($Aut(X^{\ast })$, $d$) is a metric space. Moreover, it can easily be
shown that the metric space ($Aut(X^{\ast })$, $d$) is compact.

\begin{proposition}
$Aut(X^{\ast })$ is a topological group.
\end{proposition}

\begin{proof}
First we prove that
\begin{equation*}
\begin{array}{ccccc}
\psi  & : & Aut(X^{\ast }) \times Aut(X^{\ast })& \longrightarrow  & Aut(X^{\ast }) \\
&  & (g,h) & \longmapsto  & gh %
\end{array}%
\end{equation*}%
is a continuous map. We take an arbitrary $(g_{0},h_{0})\in Aut(X^{\ast })\times Aut(X^{\ast })$. Let $U$ be a neighborhood of $g_{0}h_{0}$. There exists an integer $n$ such that
\begin{equation*}
B\Big(g_{0}h_{0},\frac{1}{p^{n}}\Big)=\Big\{f\mid d(f,g_{0}h_{0})< \frac{1}{p^{n}}\Big\}\subseteq U.
\end{equation*}
We take an open set
\begin{equation*}
V=V_{1}\times V_{2}=\{(g,h) \ | \ g\in V_{1}, h\in V_{2} \}
\end{equation*}
of $Aut(X^{\ast })\times Aut(X^{\ast })$ such that
\begin{equation*}
V_{1}=B\Big(g_{0},\frac{1}{p^{n}}\Big)=\Big\{g \ | \ d(g,g_{0})<\frac{1}{p^{n}}\Big\}
\end{equation*}
and
\begin{equation*}
V_{2}=B\Big(h_{0},\frac{1}{p^{n}}\Big)=\Big\{h \ | \ d(h,h_{0})<\frac{1}{p^{n}}\Big\}.
\end{equation*}
Now, we show that  $\psi(V)\subseteq U$ where
\begin{equation*}
\psi(V)=\psi(V_{1}\times V_{2})=\{gh \ | \ g\in V_{1}, h\in V_{2} \}.
\end{equation*}
Let $gh\in \psi(V)$. Thus, we have $g\in V_{1}$ and $h\in V_{2}$. Namely, we obtain that
\begin{equation}
\label{denk}
g^{-1}g_{0}\in St_{Aut(X^{\ast })}({n+1}) \ and \ h^{-1}h_{0}\in St_{Aut(X^{\ast })}({n+1}).
\end{equation}
Furthermore, we get
\begin{equation*}
(gh)^{-1}g_{0}h_{0}=h^{-1}(g^{-1}g_{0})h_{0}\in St_{Aut(X^{\ast })}({n+1}).
\end{equation*}
From (\ref {denk}), $gh\in U$. Thus, $\psi$ is continuous.
Similarly, we prove that
\begin{equation*}
\begin{array}{ccccc}
\varphi  & : & Aut(X^{\ast }) & \longrightarrow  & Aut(X^{\ast }) \\
&  & g & \longmapsto  & g^{-1} %
\end{array}%
\end{equation*}%
is continuous. We take an arbitrary $g_{0}\in Aut(X^{\ast })$. Let $U$ be a neighborhood of $g_{0}^{-1}$. So, there exists an integer $n$ such that
\begin{equation*}
B\Big(g_{0}^{-1},\frac{1}{p^{n}}\Big)=\Big\{f \ | \ d(f,g_{0}^{-1})<\frac{1}{p^{n}}\Big\}\subseteq U.
\end{equation*}
We take a neighborhood $V$ of $g_{0}$ in  $Aut(X^{\ast })$ such that
\begin{equation*}
V = B\Big(g_{0},\frac{1}{p^{n}}\Big)=\Big\{g \ | \ d(g,g_{0})< \frac{1}{p^{n}}\Big\}.
\end{equation*}
Now, we show that $\varphi(V)\subseteq U$. Let $g^{-1}\in \varphi(V)$. Thus, we have $g\in V$. In other words,
\begin{equation*}
gg_{0}^{-1}\in St_{Aut(X^{\ast })}({n+1}).
\end{equation*}
Due to the definition of $U$, $g^{-1}\in U$. That is, $\varphi$ is continuous.
\end{proof}

\begin{proposition}
$\overline{A}$ is a subgroup of $Aut(X^{\ast })$.
\end{proposition}

\begin{proof}
We show that  $gh\in \overline{A}$ and $g^{-1}\in \overline{A}$ for all $g,h\in \overline{A}$.

Suppose that $g,h\in \overline{A}$. This means that there are sequences $(g_{n})$, $(h_{n})$ in $A$ such that
\begin{equation*}
  \lim_{n\rightarrow \infty}g_{n}=g \ \text{and} \ \lim_{n\rightarrow \infty}h_{n}=h.
\end{equation*}
Thus, it follows that $\lim_{n\rightarrow \infty}(g_{n},h_{n})=(g,h)$. On the other hand, we proved that
\begin{equation*}
\begin{array}{ccccc}
\psi  & : & Aut(X^{\ast }) \times Aut(X^{\ast })& \longrightarrow  & Aut(X^{\ast }) \\
&  & (g,h) & \longmapsto  & gh %
\end{array}%
\end{equation*}%
is continuous. Hence, we have
\begin{equation*}
\lim_{n\rightarrow \infty}g_{n}h_{n}=gh.
\end{equation*}
It follows that $gh\in\overline{A}$ since the sequence $g_{n}h_{n}\in A$. Similarly, because
\begin{equation*}
\begin{array}{ccccc}
\varphi  & : & Aut(X^{\ast }) & \longrightarrow  & Aut(X^{\ast }) \\
&  & g & \longmapsto  & g^{-1} %
\end{array}%
\end{equation*}%
is continuous we obtain
\begin{equation*}
\lim_{n\rightarrow \infty}g_{n}^{-1}=g^{-1}.
\end{equation*}
That is, $g^{-1}\in \overline{A}$. Thus, $\overline{A}$ is a subgroup of $Aut(X^{\ast })$.
\end{proof}
\begin{center}
\renewcommand{\thefigure}{\arabic{section}.\arabic{figure}} \setcounter{figure}{0}
\renewcommand{\therc}{\arabic{section}.\arabic{rc}} \setcounter{rc}{0}

\section{Embedding of the Group of $p-$adic Integers into the Automorphism Group of the $p-ary$ Rooted Tree}
\end{center}
Now we give a formula for the distance between two elements of the adding machine group. Notice that this
expression is similar to the distance between two elements of $p-$adic integers.
\begin{proposition}
\label{psel}
For $a^{n},a^{m}\in A$, the distance $d(a^{n},a^{m})$ is
\begin{equation*}
\begin{array}{lllll}
d & : & A\times A & \rightarrow & A \\
&  &  (a^{n},a^{m}) & \mapsto & d(a^{n},a^{m})=\left\{
\begin{array}{lll}
0 &  &  \text{for} \ n=m,\\
\frac{1}{p^{k}} &  & \text{for}  \ n-m=tp^{k},%
\end{array}%
\right.%
\end{array}%
\end{equation*}%
where $t,k\in\mathbb{Z}$, $p$ is prime number and $(p,t)=1.$
\end{proposition}
\begin{proof}
First we compute $St_{A}(1)$. Using permutational wreath product we obtain that
\begin{equation*}
\begin{array}{ccl}
a^{p} & = & (1,1,\ldots,a)\sigma (1,1,\ldots,a)\sigma \ldots(1,1,\ldots,a)\sigma  \\
& = & (a,a,\ldots,a).%
\end{array}%
\end{equation*}
Thus, $St_{A}(1)=\langle a^{p} \rangle$.
Moreover, we get
\begin{equation*}
\begin{array}{ccl}
a^{p^{2}} & = & a^{p}a^{p}\ldots a^{p} \\
& = & (a,a,\ldots,a)(a,a,\ldots,a)\ldots(a,a,\ldots,a) \\
& = & (a^{p},a^{p},\ldots,a^{p})%
\end{array}%
\end{equation*}
We have $a^{p^{2}}\in St_{A}(2)$ because $a^{p}\in St_{A}(1)$. Therefore, $St_{A}(2)=\langle a^{p^{2}} \rangle$. By proceeding in a similar manner, we compute $St_{A}(k)=\langle a^{p^{k}} \rangle$.

So, elements of $A$ which are in $St_{A}(1)$ but are not in $St_{A}(2)$ can be expressed as
\begin{equation*}
St_{A}(1)-St_{A}(2)=\{a^{tp}:(p,t)=1 \}
\end{equation*}
and in general, we have
\begin{equation*}
St_{A}(k)-St_{A}(k+1)=\{a^{tp^{k}}:(p,t)=1\}.
\end{equation*}
Let us take arbitrary $a^{n},a^{m}\in A$. If $n=m$ then $a^{n}=a^{m}$ and $d(a^{n},a^{m})=0$.
Assume $n\neq m$. So there is a unique expression $n-m=tp^{k}$ such that $(p,t)=1$. Then we obtain
\begin{equation*}
a^{-m}a^{n}=a^{n-m}=a^{tp^{k}}\in St_{A}(k)-St_{A}(k+1)
\end{equation*}
and $d(a^{n},a^{m})=\frac{1}{p^{k}}$.
\end{proof}

\begin{proposition}
\label{yak}
Let $\sum_{i \geq 0} \alpha_{i}p^{i}\in \mathbb{Z}_{p}$. Then the sequence
\begin{equation*}
a^{\alpha_{0}}, a^{\alpha_{0}+\alpha_{1}p},
a^{\alpha_{0}+\alpha_{1}p+\alpha_{2}p^{2}},\ldots
\end{equation*}
is convergent.
\end{proposition}
\begin{proof}
For any $\varepsilon>0$, there is a positive integer $n_{0}$ such that $\frac{1}{p^{n_{0}}}<\varepsilon$. If $k>l$ and $k,l\geq n_{0}$ then it is obtained that
\begin{equation*}
d(a^{\alpha_{0}+\alpha_{1}p+...+\alpha_{k}p^{k}},a^{\alpha_{0}+\alpha_{1}p+\ldots+\alpha_{l}p^{l}})=\frac{1}{p^{l}}<\varepsilon.
\end{equation*}
from Proposition \ref{psel}. Thus, it is a Cauchy sequence. Because $Aut(X^{\ast })$ is a complete metric space, this sequence is convergent.
\end{proof}
Now we give our main proposition:
\begin{proposition}
{\normalsize \label{gomme} }
We define
{\normalsize
\begin{equation*}
\begin{array}{ccccc}
\varphi & : & \mathbb{Z}_{p} & \rightarrow & \overline{A} \\
\end{array}%
\end{equation*}
such that $\varphi(\sum_{i \geq 0} \alpha_{i}p^{i})$ is the limit of the sequence $a^{\alpha_{0}}, a^{\alpha_{0}+\alpha_{1}p},
a^{\alpha_{0}+\alpha_{1}p+\alpha_{2}p^{2}},\ldots$. Then $\varphi$ is both an isometry and a group isomorphism.}
\end{proposition}

{\normalsize
\begin{proof}
From Proposition \ref{yak}, $\varphi$ is well-defined.
Now we show that $\varphi$ is an isometry. In other words, we show that $d_{p}(\alpha, \beta)=d(\varphi(\alpha), \varphi(\beta))$ for every $ \alpha, \beta \in \mathbb{Z}_{p}$. Let $\alpha=\sum_{i \geq 0} \alpha_{i}p^{i}$ and $\beta=\sum_{i \geq 0} \beta_{i}p^{i}$.

If $d_{p}(\alpha, \beta)=0$ then we obtain $d(\varphi(\alpha), \varphi(\beta))=0$ since $\alpha_{i}=\beta_{i}$ for $i=0,1,2,\ldots$.

If $d_{p}(\alpha, \beta)=\frac{1}{p^{k}}$ then $\alpha_{i}=\beta_{i}$ for $i< k$ and $\alpha_{k}\neq \beta_{k}$. We must show that $d(\varphi(\alpha), \varphi(\beta))=\frac{1}{p^{k}}$. Because $\varphi(\alpha)$ and $\varphi(\beta)$ are the limits of the sequences $a^{\alpha_{0}}, a^{\alpha_{0}+\alpha_{1}p}, a^{\alpha_{0}+\alpha_{1}p+\alpha_{2}p^{2}},\ldots$ and  $a^{\beta_{0}}, a^{\beta_{0}+\beta_{1}p}, a^{\beta_{0}+\beta_{1}p+\beta_{2}p^{2}},\ldots$ respectively, it is obtained that
\begin{equation*}
\lim_{k\rightarrow\infty}(a^{\alpha_{0}+\alpha_{1}p+...+\alpha_{k}p^{k}},a^{\beta_{0}+\beta_{1}p+...+\beta_{k}p^{k}})=(\varphi(\alpha),\varphi(\beta)).
\end{equation*}
Since any metric function is continuous,
\begin{equation*}
d(a^{\alpha_{0}},a^{\beta_{0}}), d(a^{\alpha_{0}+\alpha_{1}p},a^{\beta_{0}+\beta_{1}p}),\ldots\rightarrow d(\varphi(\alpha),\varphi(\beta)).
\end{equation*}
From Proposition \ref{psel}, we get
\begin{equation*}
0,0,...,0,\frac{1}{p^{k}},\frac{1}{p^{k}},\ldots,\frac{1}{p^{k}},\ldots\rightarrow \frac{1}{p^{k}}.
\end{equation*}
So, we get $d(\varphi(\alpha),\varphi(\beta))=\frac{1}{p^{k}}$. Namely, $\varphi$ is an isometry map.

Moreover, $\varphi$ is injective since $\varphi$ is an isometry map.

Now we show that $\varphi$ is surjective. Let $b\in \overline{A}$ be arbitrary. Thus, there exists a sequence
\begin{equation*}
\label{dizi}
a^{n_{0}},a^{n_{1}},\ldots,a^{n_{k}},\ldots\rightarrow b
\end{equation*}
whose elements are in $A$. Furthermore, every integer $n_{k}$ can be expressed in $\mathbb{Z}_{p}$ as
\begin{equation}
\label{i}
\begin{array}{ccc}
n_{0} & = & \alpha _{0}^{0}+\alpha _{1}^{0}p+\alpha _{2}^{0}p^{2}+\ldots \\
n_{1} & = & \alpha _{0}^{1}+\alpha _{1}^{1}p+\alpha _{2}^{1}p^{2}+\ldots \\
  & \vdots &  \\
n_{k} & = & \alpha _{0}^{k}+\alpha _{1}^{k}p+\alpha _{2}^{k}p^{2}+\ldots \\
 & \vdots &
\end{array}%
\end{equation}
At least one of the numbers $0,1,2,...,(p-1)$ occurs infinitely many times in the sequence $(\alpha_{0}^{k})_{k}$. We choose one of them and denote it by $\beta_{0}$. Let $(\alpha_{1}^{k_{l}})_{l}$ be a subsequence of $(\alpha_{1}^{k})_{k}$ such that $\alpha_{0}^{k_{l}}=\beta_{0}$ for $l=0,1,2,\ldots$. Similarly, we denote by $\beta_{1}$, any one of the numbers that appears infinitely many times in the sequence $(\alpha_{1}^{k_{l}})_{l}$.
Proceeding in this manner, we obtain a sequence
\begin{equation*}
\label{dizi2}
a^{\beta_{0}},a^{\beta_{0}+\beta_{1}p},\ldots,a^{\beta_{0}+\beta_{1}p+\ldots+\beta_{k}p^{k}},\ldots.
\end{equation*}
From Proposition \ref{yak}, this sequence is convergent. Now we show this sequence converges to $b$. Due to the construction of (\ref{i}), there exists a subsequence $(n_{k_{s}})$ of the sequence $({n_{k}})$ whose $p-$adic expression of term $s$th such that
\begin{equation*}
\beta_{0}+\beta_{1}p+\beta_{2}p^{2}+\ldots+\beta_{s}p^{s}+\gamma_{s+1}p^{s+1}+\gamma_{s+2}p^{s+2}+\ldots
\end{equation*}
Hence, because
\begin{equation*}
\lim_{s\rightarrow\infty}d(a^{\beta_{0}+\beta_{1}p+\ldots+\beta_{s}p^{s}},a^{n_{k_{s}}})=0
\end{equation*}
and from the triangle inequality, the sequence $(a^{\beta_{0}+\beta_{1}p+\ldots+\beta_{k}p^{k}})$ converges to $b$. So, $\varphi(\sum_{i\geq0}\beta_{i}p^{i})=b$ and $\varphi$ is surjective.

Finally, we prove that $\varphi$ is a homomorphism. In other words, we prove that
\begin{equation*}
\varphi(\alpha+\beta)=\varphi(\alpha)\varphi(\beta)
\end{equation*}
for every $\alpha,\beta\in\mathbb{Z}_{p}$. Let $\alpha=\alpha _{0}+\alpha _{1}p+\alpha _{2}p^{2}+\ldots$, $\beta=\beta_{0}+\beta_{1}p+\beta_{2}p^{2}+\ldots$ and
\begin{equation*}
\alpha+\beta=\gamma_{0}+\gamma_{1}p+\gamma_{2}p^{2}+\ldots.
\end{equation*}
From the definition of $\varphi$,
\begin{equation*}
  a^{\gamma_{0}},a^{\gamma_{0}+\gamma_{1}p},a^{\gamma_{0}+\gamma_{1}p+\gamma_{2}p^{2}},...\rightarrow \varphi(\alpha+\beta).
\end{equation*}
Moreover, it follows that
\begin{equation*}
 a^{(\alpha_{0}+\beta_{0})},a^{(\alpha_{0}+\beta_{0})+(\alpha_{1}+\beta_{1})p},a^{(\alpha_{0}+\beta_{0})+(\alpha_{1}+\beta_{1})p+(\alpha_{2}+\beta_{2})p^{2}},\ldots\rightarrow \varphi(\alpha)\varphi(\beta)
\end{equation*}
since $Aut(X^{*})$ is a topological group,
\begin{equation*}
 a^{\alpha_{0}},a^{\alpha_{0}+\alpha_{1}p},a^{\alpha_{0}+\alpha_{1}p+\alpha_{2}p^{2}},\ldots\rightarrow \varphi(\alpha)
\end{equation*}
and
\begin{equation*}
 a^{\beta_{0}},a^{\beta_{0}+\beta_{1}p},a^{\beta_{0}+\beta_{1}p+\beta_{2}p^{2}},\ldots\rightarrow \varphi(\beta).
\end{equation*}
In $\mathbb{Z}_{p}$,
\begin{equation*}
\begin{array}{cclcl}
  \alpha _{0}+\beta_{0} & = & \gamma _{0}+ \overline{\gamma_{0}}p+0p^{2}+0p^{3}+\ldots \\
 \alpha _{0}+\beta_{0} + (\alpha _{1}+\beta_{1})p & = & \gamma _{0}+\gamma _{1}p+\overline{\gamma_{1}}p^{2}+0p^{3}+0p^{4}+\ldots \\
  &  \vdots  &  & &   \\
  \alpha _{0}+\beta_{0} +\ldots+(\alpha _{k}+\beta_{k})p^{k} & = & \gamma _{0}+\gamma _{1}p+\ldots+\gamma _{k}p^{k}+\overline{\gamma_{k}}p^{k+1}+0p^{k+2}+0p^{k+3}+\ldots.\\
  &\vdots & & &
\end{array}%
\end{equation*}
Let $x= \alpha _{0}+\beta_{0} +\ldots+(\alpha _{k}+\beta_{k})p^{k}$ and $y=\gamma _{0}+\gamma _{1}p+\ldots+\gamma _{k}p^{k}+\overline{\gamma_{k}}p^{k+1}+0p^{k+2}+0p^{k+3}+\ldots$. Then we have
\begin{equation*}
d(a^{x}, a^{y})=\left\{
\begin{array}{lll}
\frac{1}{p^{k}} &  &\text{if} \ \overline{\gamma_{k}}\neq 0, \\
0 &  &\text{if} \ \overline{\gamma_{k}}=0.%
\end{array}%
\right.
\end{equation*}%
So we get $\varphi(\alpha+\beta)=\varphi(\alpha)\varphi(\beta)$ since
\begin{equation*}
d(a^{\alpha _{0}+\beta_{0}},a^{\gamma _{0}}),d(a^{\alpha _{0}+\beta_{0} + (\alpha _{1}+\beta_{1})p},a^{\gamma _{0}+\gamma _{1}p}),\ldots\rightarrow d(\varphi(\alpha)\varphi(\beta), \varphi(\alpha+\beta))
\end{equation*}
and
\begin{equation*}
\lim_{k\rightarrow \infty}d(a^{x}, a^{y})=0.
\end{equation*}
Thus the proof is completed.
\end{proof}
}
Consequently, the group of $p-$adic integers $\mathbb{Z}_{p}$ can be isometrically embedded into the metric space $Aut(X^{*})$ since $\overline{A}\subseteq Aut(X^{*})$.
\begin{example}
We show $\varphi(-1)$ for $p=2$ in Figure \ref{aaa}. It is well-known that
\begin{equation*}
-1=1+1.2^{1}+1.2^{2}+...+1.2^{k}+...\in \mathbb{Z}_{2}.
\end{equation*}
Due to the definition of $\varphi$, $\varphi(-1)$ is the limit of the sequence
\begin{equation*}
a^{1},a^{1+1.2^{1}},a^{1+1.2^{1}+1.2^{2}},...
\end{equation*}
in $A$ for $X=\{0,1\}$. This limit equals to $a^{-1}=(a^{-1},1)\sigma$ because of Proposition \ref{psel}.

\begin{figure}[h]

\centering \includegraphics[scale=0.70]{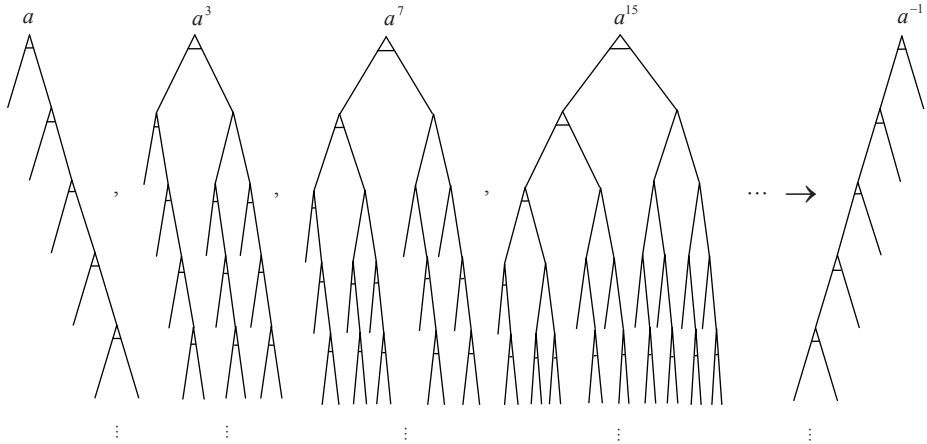}
\caption{The image of $-1\in \mathbb{Z}_{2}$ under the map $\varphi$}
\label{aaa}
\end{figure}
\end{example}

\end{document}